\documentclass[
final
]{dmtcs-episciences}

\usepackage[utf8]{inputenc}
\usepackage{subfigure}

\usepackage[utf8]{inputenc}
\usepackage{subfigure}
\usepackage{amssymb,amsfonts,amsmath, psfrag,eepic,colordvi,graphicx,epsfig,ytableau}
\usepackage[enableskew]{youngtab}

\newcommand{\rmnum}[1]{\romannumeral #1}
 \numberwithin{equation}{section}
\newtheorem{theorem}{Theorem}[section]

\newtheorem{conjecture}[theorem]{Conjecture}

\newtheorem{lemma}[theorem]{Lemma}

\usepackage[round]{natbib}

\author{Sherry H.F. Yan\affiliationmark{1} \thanks{She is fully supported by    the National Natural Science Foundation of China (11671366 and  11571320) and  Zhejiang Provincial Natural Science Foundation of China ( LY15A010008 ).}
  \and Robin D.P. Zhou\affiliationmark{2}\thanks{He is fully supported by  the National Natural Science Foundation of China (11671366  and  11626158) and  Zhejiang Provincial Natural Science Foundation of China ( LQ17A010004).}
   }
\title{Refined   Enumeration of  Corners in   Tree-like Tableaux}
\affiliation{
  Department of Mathematics, Zhejiang Normal University,   P.R. China\\
  College of Mathematics Physics and Information, Shaoxing University,  P.R. China  }
\received{2017-5-24}
\revised{2017-9-4}
\accepted{2017-9-13}
\begin{document}
\publicationdetails{19}{2017}{3}{6}{3683}
\maketitle
\begin{abstract}
 Tree-like tableaux are certain fillings of Ferrers diagrams originally introduced by Aval et al.,  which are in simple bijections with   permutation tableaux  coming from  Postnikov's study of totally nonnegative Grassmanian and  alternative tableaux  introduced by   Viennot.
In this paper, we confirm two conjectures of Gao et al.  on the  refined  enumeration of non-occupied corners in  tree-like tableaux and symmetric tree-like tableaux  via   intermediate structures of   alternative tableaux,  linked partitions, type $B$ alternative tableaux and type $B$ linked partitions.
\end{abstract}

\section{Introduction}
\label{sec:in}

This paper is concerned with two conjectures of  \cite{Gao} on the  refined  enumeration of tree-like tableaux and symmetric tree-like tableaux with respect to number of non-occupied corners.
Tree-like tableaux are certain fillings of Ferrers diagrams originally introduced by   \cite{aval},  which are in simple bijections with   permutation tableaux  coming from  Postnikov's study of totally nonnegative Grassmanian in \cite{Pos}  and  alternative tableaux  introduced by    \cite{Vi}. These three equivalent combinatorial objects are closely related to a statistical physics model called partially asymmetric exclusion process (PASEP); see \cite{Cor0, Cor3, Cor4}.

  A  {\em Ferrers diagram } is the
left-justified arrangement of square cells with possibly empty rows and columns. The {\em size
 } of  a Ferrers diagram is  the number of rows plus the number of columns. Given a Ferrers diagram of  size $n$, we label the steps in the south-east border with $1,2,\ldots, n$ from north-east to south-west. A row (resp. column) is labeled with $i$ if the row (resp. column ) contains the south (resp. west) step labeled with $i$. Notice that we may place a row label to the left of the first column and place a column label at the top of the first row; see Figure \ref{diagram}.
  A row (resp. column) labeled with $i$ is called row (resp. column) $i$.
    The steps of the south-east border are called {\em border edges}.
  The cell $(i,j)$ is the cell in row $i$ and column $j$ unless otherwise stated.

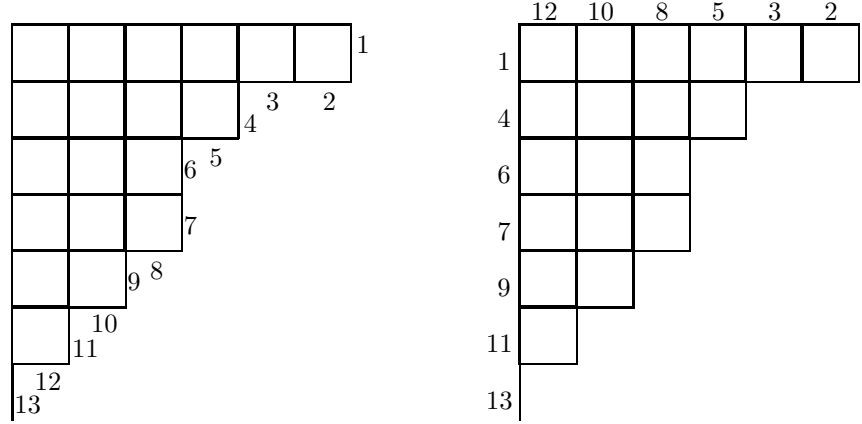
\begin{figure}[h!]
\begin{center}
\begin{picture}(200,200)
\setlength{\unitlength}{1.5mm}
\put(0,0){\line(0,1){5} }
\put(0,5){\framebox(5,5) }
\put(0,10){\framebox(5,5)   } \put(5,10){\framebox(5,5)  }
\put(0,15){\framebox(5,5)   } \put(5,15){\framebox(5,5)  }\put(10,15){\framebox(5,5)  }

\put(0,20){\framebox(5,5)   } \put(5,20){\framebox(5,5)  }\put(10,20){\framebox(5,5)  }

\put(0,25){\framebox(5,5)   } \put(5,25){\framebox(5,5)  }
\put(10,25){\framebox(5,5)  }\put(15,25){\framebox(5,5)  }

\put(0,30){\framebox(5,5)   } \put(5,30){\framebox(5,5)  }
\put(10,30){\framebox(5,5)  }\put(15,30){\framebox(5,5)  }
\put(20,30){\framebox(5,5)  }\put(25,30){\framebox(5,5)  }

\put(30.5, 32.5){$1$}\put(27.5, 27.5){$2$}\put(22.5, 27.5){$3$}
\put(20.5, 25.5){$4$}\put(17.5, 22.5){$5$}\put(15.2, 21.5){$6$}
\put(15.2, 16.5){$7$}\put(12.2, 12.5){$8$}\put(10.2, 11.5){$9$}
\put(7, 7.8){$10$}\put(5.3, 5.6){$11$}\put(2, 2.6){$12$}\put(0.2, 0.6){$13$}

\put(45,0){\line(0,1){5} }
\put(45,5){\framebox(5,5) }
\put(45,10){\framebox(5,5)   } \put(50,10){\framebox(5,5)  }
\put(45,15){\framebox(5,5)   } \put(50,15){\framebox(5,5)  }\put(55,15){\framebox(5,5)  }

\put(45,20){\framebox(5,5)   } \put(50,20){\framebox(5,5)  }\put(55,20){\framebox(5,5)  }

\put(45,25){\framebox(5,5)   } \put(50,25){\framebox(5,5)  }
\put(55,25){\framebox(5,5)  }\put(60,25){\framebox(5,5)  }

\put(45,30){\framebox(5,5)   } \put(50,30){\framebox(5,5)  }
\put(55,30){\framebox(5,5)  }\put(60,30){\framebox(5,5)  }
\put(65,30){\framebox(5,5)  }\put(70,30){\framebox(5,5)  }

\put(42, 1){$13$}\put(42, 6){$11$}\put(43, 11){$9$}\put(43, 16){$7$}
\put(43, 21){$6$}\put(43, 26){$4$}\put(43, 31){$1$}
\put(46, 35.5){$12$}\put(51, 35.5){$10$}\put(57, 35.5){$8$}\put(62, 35.5){$5$}
\put(67, 35.5){$3$}\put(72, 35.5){$2$}
  \end{picture}
\end{center}
\caption{  The labelling of a Ferrers diagram.}\label{diagram}
\end{figure}

 A {\em tree-like  } tableau is a  filling of a Ferrers diagram without empty rows and columns with points inside some cells, such that the resulting diagram    satisfies the following conditions:
 \begin{itemize}
 \item[(1)]  the top left cell contains a point, called the root point;
 \item[(2)] for every non-root pointed cell $c$, there exists either a pointed cell above $c$ in the same column, or a pointed cell to its left in the same row, but not both;
     \item[(3)] every row and every column possess at least one pointed cell.
 \end{itemize}

The size of a  tree-like tableau is defined to be its number of points.
The left  subfigure of Figure \ref{tree} illustrates a tree-like tableau of size $11$.
It is well known that the size of a tree-like tableau is equal to  the size of its underlying diagram  minus  one.
Denote by $\mathcal{T}_n$ the set of tree-like tableaux of size $n$.
A {\em symmetric tree-like tableau} is a tree-like tableau unchanged by the reflection with respect to its main diagonal line; see the right   subfigure of Figure \ref{tree}.  Since  the size of a symmetric tree-like tableau is necessarily odd, we denote by $\mathcal{T}^{sym}_{2n+1}$ the set of symmetric tree-like tableaux of size $2n+1$.

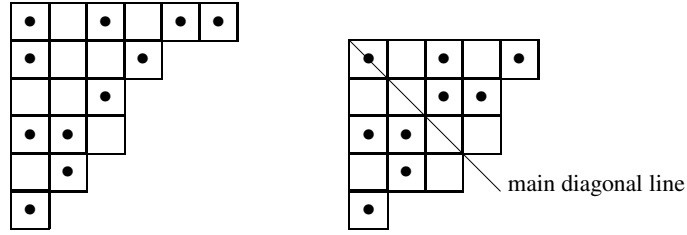
\begin{figure}[h!]
\begin{center}
\begin{picture}(150,100)
\setlength{\unitlength}{1mm}

\put(0,0){\framebox(5,5)  {$\bullet$} }
\put(0,5){\framebox(5,5)   } \put(5,5){\framebox(5,5)   {$\bullet$} }
\put(0,10){\framebox(5,5)    {$\bullet$} } \put(5,10){\framebox(5,5)    {$\bullet$} } \put(10,10){\framebox(5,5)  }

\put(0,15){\framebox(5,5)   } \put(5,15){\framebox(5,5)  }\put(10,15){\framebox(5,5)  {$\bullet$}  }

\put(0,20){\framebox(5,5)   {$\bullet$}  } \put(5,20){\framebox(5,5)  }
\put(10,20){\framebox(5,5)  }\put(15,20){\framebox(5,5) {$\bullet$}   }

\put(0,25){\framebox(5,5)  {$\bullet$}   } \put(5,25){\framebox(5,5)  }
\put(10,25){\framebox(5,5) {$\bullet$}   }\put(15,25){\framebox(5,5)  }
\put(20,25){\framebox(5,5) {$\bullet$}   }\put(25,25){\framebox(5,5)  {$\bullet$}  }

 \put(45,0){\framebox(5,5)  {$\bullet$} }
\put(45,5){\framebox(5,5)   } \put(50,5){\framebox(5,5)   {$\bullet$} }
\put(55,5){\framebox(5,5)     }

\put(45,10){\framebox(5,5)  {$\bullet$} }
\put(50,10){\framebox(5,5) {$\bullet$}  } \put(55,10){\framebox(5,5)   }
\put(60,10){\framebox(5,5)     }

\put(45,15){\framebox(5,5)   }
\put(50,15){\framebox(5,5)   } \put(55,15){\framebox(5,5) {$\bullet$}  }
\put(60,15){\framebox(5,5)    {$\bullet$} }

\put(45,20){\framebox(5,5) {$\bullet$}  }
\put(50,20){\framebox(5,5)   } \put(55,20){\framebox(5,5) {$\bullet$}  }
\put(60,20){\framebox(5,5)     }
\put(65,20){\framebox(5,5) {$\bullet$}  }

\put(45,25){\line(1,-1){20}} \put(66,5){\small main diagonal line}

  \end{picture}
\end{center}
\caption{  A tree-like tableau (left) and a symmetric tree-like tableau (right).}\label{tree}
\end{figure}

In a  Ferrers diagram, a {\em corner} is a cell such that its bottom and right edges are border edges. If the corner $c$ is not pointed in a tree-like tableau $T$, then we say that  $c$ is a {\em non-occupied } corner of $T$.  Recently,  \cite{Lab} noticed that the corners in tree-like tableaux  could be interpreted in the PASEP as the locations where a jump of particle is possible.  In \cite{Lab}, he also obtained the enumeration of occupied corners in  tree-like tableaux and symmetric tree-like tableaux and posed two conjectures concerning the total number of corners in   tree-like tableaux and symmetric tree like tableaux. These two conjectures were independently confirmed by    \cite{Gao},  and
 \cite{Hit}.

 Given a tree-like tableau $T$, the {\em weight} of $T$  is defined to be $a^{top(T)}b^{left(T)}$, where $top(T)$ is the number of non-root pointed cells in the topmost row, and  $left(T)$
 is the number of non-root pointed cells in the  leftmost column.
The  polynomial analogues  of the  number of  tree-like tableaux  of size $n$
and symmetric  tree-like tableaux  of size $2n+1$
are  defined as follows:
\begin{align*}
T_n(a,b)=\sum_{T\in \mathcal{T}_n}a^{top(T)}b^{left(T)},\\
T^{sym}_{2n+1}(a,b)=\sum_{T\in \mathcal{T}^{sym}_{2n+1}}x^{left(T)-1}.
\end{align*}
By introducing an elementary insertion procedure,   \cite{aval} showed that
\begin{equation}\label{eqT}
T_n(a,b)=\sum_{T\in \mathcal{T}_n}a^{top(T)}b^{left(T)}=(a+b)_{n-1},
\end{equation}
and

\begin{equation}\label{eqsT}
  T^{sym}_{2n+1}(x)=\sum_{T\in \mathcal{T}^{sym}_{2n+1}}x^{left(T)-1}=2^n(x+1)_{n-1},
  \end{equation}
where  $(x)_n$ denotes the rising factorial, that is,  $(x)_0=1$ and $(x)_n=x(x+1)\ldots (x+n-1)$ for $n\geq 1$.
In order to refine the enumeration of  non-occuppied corners in tree-like tableaux and symmetric tree like tableaux,  \cite{Gao} posed the following two conjectures.

\begin{conjecture}{ \upshape   (See \cite{Gao}, Conjecture 4.10)}\label{con1}
For  $n\geq 3$, the $(a,b)$-analogue of the number of non-occupied corners in tree-like tableaux of size $n$ is given by
$$
noc_n(a,b)=\sum_{T\in \mathcal{T}_n}noc(T)a^{top(T)}b^{left(T)}=\big( (n-2)ab+{n-2\choose 2}(a+b)+{n-2\choose 3}\big)\cdot  T_{n-2}(a,b)
$$
 where $noc(T)$ is the number of non-occupied corners of $T$.
\end{conjecture}

 \begin{conjecture}{ \upshape   (See \cite{Gao}, Conjecture 4.14)}\label{con2}
 For $n\geq 3$, the $x$-analogue of the number of non-occupied corners in symmetric tree-like tableaux of size $2n+1$ is given by
$$
\sum_{T\in \mathcal{T}^{sym}_{2n+1}}noc(T)x^{left(T)-1}=\big[2nx^2+2(2n^2-4n+1)x+{(n-2)(n-1)(4n-3)\over 3}\big]\cdot T^{sym}_{2n-3}(x).
$$
 \end{conjecture}

To confirm  Conjecture \ref{con1}, we establish a bijection  between tree-like tableaux of size $n$ and alternative tableaux of size $n$ in which each column contains an up arrow, and    a bijection between  alternative tableaux of size $n$ in which each column contains an up arrow and linked partitions of $[n]$. The notion of linked partitions was introduced by  \cite{Dy}  in the study of the unsymmetrized T-transform in free probability theory.  The combination of these two bijections enables us to compute the left-hand side of Conjecture \ref{con1} in terms of linked partitions.

In order to verify Conjecture \ref{con2}, we introduce the notions of type $B$ alternative tableaux and  type $B$ linked partitions. We establish a bijection between symmetric alternative tableaux of size $2n$ and type $B$ alternative tableaux of  size $n$, and      a bijection between type $B$ alternative tableaux of  size $n$ and  type $B$ linked partitions of $[n]$.  By combining    these two bijections and  the bijection between symmetric tree-like tableaux of size $2n+1$ and symmetric alternative tableaux of size $2n$ given by    \cite{aval}, we can    compute the left-hand side of Conjecture \ref{con2} in terms of type $B$ linked partitions.

\section{The bijection between tree-like tableaux  and alternative tableaux}
\label{sec:bi}
In this section,  we give an overview  of the bijection $\alpha$
between tree-like tableaux of size $n$ and alternative tableaux of size
$n-1$ which was established by   \cite{aval}.
Analogous to the bijection $\alpha$, we establish a bijection between
 tree-like tableaux of size $n$ and alternative tableaux of size
$n$ in which each column contains an up arrow.
To prove Conjecture \ref{con2}, we introduce the notion of
type $B$ alternative tableaux and construct a bijection between
type $B$ alternative tableaux of size $n$
and symmetric alternative tableaux of size $2n$

An {\em alternative } tableau is a     Ferrers diagram with a partial filling of the cells with left arrows $\leftarrow$ and up arrows $\uparrow$, such that all the cells to the left of  a left arrow  $\leftarrow$, or above an up arrow  $\uparrow$ are empty. In other words, all the cells pointed by an arrow must be empty.  Alternative tableaux were first introduced by   \cite{Vi} and systematically studied by \cite{Nad}.
  The size of an alternative tableau is defined to be the size
  of its underlying diagram.
  We denote by $\mathcal{AT}_n$ the set of alternative tableaux of size $n$.
   In an alternative tableau $T$, a row is said to be {\em unrestricted} if it has no $\leftarrow's$.
  A corner $c$ is said to be {\em non-occupied} if the cell $c$ is empty. Denote by $urr(T)$ and $noc(T)$ the number of unrestricted rows and the number of non-occupied corners of $T$, respectively.

   A {\em symmetric } alternative tableau is an alternative  tableau unchanged by the reflection with respect to its main diagonal line.
 It is obvious that the size of a symmetric alternative tableau is even.
   Denote by  $\mathcal{AT}^{sym}_{2n}$ the set of symmetric alternative tableaux of size $2n$.

Recall that $\mathcal{T}_n$ denotes the set of tree-like tableaux of size $n$. In \cite{aval}, they established a  bijection $\alpha$ between the set $\mathcal{T}_n$ and the set $\mathcal{AT}_{n-1}$.
In the following, we give  a description of the bijection $\alpha$
and its inverse $\alpha^{-1}$ without proof, see \cite{aval}
for more details.

Given a tree-like tableau $T\in \mathcal{T}_n$ with underlying diagram $F$, we first replace each non-root point $p$  with an $\uparrow$ if there is no point above $p$ in the same column and with a $\leftarrow$ if there is no point to the left of $p$ in the same row. Then, we remove the topmost row and the leftmost column from $F$. Let $\alpha(T)$ be the resulting filling.
For example, Figure \ref{alternative} illustrates  a tree-like tableau $T$ and its corresponding alternative tableau $\alpha(T)$.
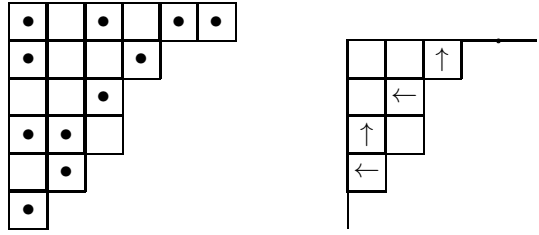
\begin{figure}[h!]
\begin{center}
\begin{picture}(100,100)
\setlength{\unitlength}{1mm}

\put(0,0){\framebox(5,5)  {$\bullet$} }
\put(0,5){\framebox(5,5)   } \put(5,5){\framebox(5,5)   {$\bullet$} }
\put(0,10){\framebox(5,5)    {$\bullet$} } \put(5,10){\framebox(5,5)    {$\bullet$} } \put(10,10){\framebox(5,5)  }

\put(0,15){\framebox(5,5)   } \put(5,15){\framebox(5,5)  }\put(10,15){\framebox(5,5)  {$\bullet$}  }

\put(0,20){\framebox(5,5)   {$\bullet$}  } \put(5,20){\framebox(5,5)  }
\put(10,20){\framebox(5,5)  }\put(15,20){\framebox(5,5) {$\bullet$}   }

\put(0,25){\framebox(5,5)  {$\bullet$}   } \put(5,25){\framebox(5,5)  }
\put(10,25){\framebox(5,5) {$\bullet$}   }\put(15,25){\framebox(5,5)  }
\put(20,25){\framebox(5,5) {$\bullet$}   }\put(25,25){\framebox(5,5)  {$\bullet$}  }

\put(45,0){\line(0,1){5}}
\put(45,5){\framebox(5,5)  {$\leftarrow$} }
\put(45,10){\framebox(5,5)  {$\uparrow$} }
 \put(50,10){\framebox(5,5)    }

 \put(45,15){\framebox(5,5)    }
 \put(50,15){\framebox(5,5)  {$\leftarrow$}  }

 \put(45,20){\framebox(5,5)    }

 \put(50,20){\framebox(5,5)    }
 \put(55,20){\framebox(5,5)  {$\uparrow$}  }

 \put(60,25){\line(1,0){5}}\put(65,25){\line(1,0){5}}
 \put(65,25){\circle*{0.5}}
  \end{picture}
\end{center}
\caption{  A tree-like tableau $T$ (left) and  its corresponding alternative tableaux  tableau  $\alpha(T)$ (right).}\label{alternative}
\end{figure}

The inverse map $\alpha^{-1}$ is defined as follows.  Let $T'$   be an alternative tableau of size  $n-1$ with underlying diagram $F'$.  Suppose that $F'$ has $k$ rows and $n-1-k$ columns.  We first construct the underlying diagram $F$ of  $\alpha^{-1}(T')$ from $F'$ by adding a column of $k$ cells to the left of the leftmost column and adding a row of $n-1-k$ cells above its topmost row and a cell at its top left corner. It is easily seen that $F$ is of size $n+1$.
Next,   we fill the leftmost cell in  row $i$ of $F$ with a point if and only if   row $i-1$ of $F'$ has no $\leftarrow's$.
 Analogously,    we fill  the topmost  cell in column $j$ of $F$
 with a point if and only if   column $j-1$ of $F'$ has no $\uparrow's$.
 Finally, we  fill each cell $(i+1,j+1)$ with a point if and only if the cell $(i,j)$  is filled with an arrow in $F'$
 and  fill a point in the top left cell of $F$.
 Let $\alpha^{-1}(T')$ be the resulting tableau.

From the definition of the bijection $\alpha$, it is not difficult to see that $\alpha$ maps a non-root point in the first column  of $T$ to an unrestricted row of $T'$. Moreover, the bijection $\alpha$ preserves the number of non-occupied corners since every row and every column of $T$ possess at least one pointed cell.
 Hence, we have the following properties of the bijection $\alpha$.

\begin{theorem}\label{thealpha}
Let $n\geq 1$. For any $T\in \mathcal{T}_n$ and $T'\in \mathcal{AT}_{n-1}$ with $\alpha(T)=T'$, we have  $left(T)=urr(T')$ and $noc(T)=noc(T')$.
\end{theorem}

 \begin{theorem}\label{thesym}
 Let $n\geq 1$. The bijection $\alpha$ restricted to the set of symmetric tree-like tableaux of size $2n+1$ induces a bijection between the set   $\mathcal{T}^{sym}_{2n+1}$ and the set  $\mathcal{AT}^{sym}_{2n}$. Furthermore,
for any $T\in \mathcal{T}^{sym}_{2n+1}$ and $T'\in \mathcal{AT}^{sym}_{2n}$ with $\alpha(T)=T'$, we have
$left(T)=urr(T')$ and
    $noc(T)=noc(T')$.
   \end{theorem}

In order to prove Conjecture \ref{con1}, we need to consider a subset of $T\in \mathcal{AT}_n$, denoted by $\mathcal{AT}^{*}_n$,  in which each column of $T$ contains an $\uparrow$.
  Analogous to the bijection $\alpha$, we establish  a bijection $\beta$ between the set $\mathcal{T}_n$ to the set $\mathcal{AT}^{*}_{n}$ without proof.

 Given a tree-like tableau $T$ of size $n$ with underlying diagram $F$, we construct an alternative tableau $T'=\beta(T)$ as follows. First replace each non-root point $p$  with an $\uparrow$ if there is no point above $p$ in the same column and with a $\leftarrow$ if there is no point to the left of $p$. Then,
 $T'$ is the resulting tableau by removing the leftmost column from $F$.
  For example, Figure \ref{alternative1} illustrates  a tree-like tableau $T$ and its corresponding alternative tableau $\beta(T)$.

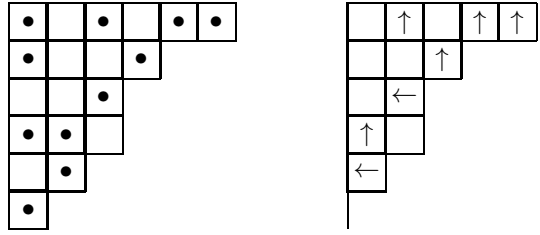
\begin{figure}[h!]
\begin{center}
\begin{picture}(100,100)
\setlength{\unitlength}{1mm}

\put(0,0){\framebox(5,5)  {$\bullet$} }
\put(0,5){\framebox(5,5)   } \put(5,5){\framebox(5,5)   {$\bullet$} }
\put(0,10){\framebox(5,5)    {$\bullet$} } \put(5,10){\framebox(5,5)    {$\bullet$} } \put(10,10){\framebox(5,5)  }

\put(0,15){\framebox(5,5)   } \put(5,15){\framebox(5,5)  }\put(10,15){\framebox(5,5)  {$\bullet$}  }

\put(0,20){\framebox(5,5)   {$\bullet$}  } \put(5,20){\framebox(5,5)  }
\put(10,20){\framebox(5,5)  }\put(15,20){\framebox(5,5) {$\bullet$}   }

\put(0,25){\framebox(5,5)  {$\bullet$}   } \put(5,25){\framebox(5,5)  }
\put(10,25){\framebox(5,5) {$\bullet$}   }\put(15,25){\framebox(5,5)  }
\put(20,25){\framebox(5,5) {$\bullet$}   }\put(25,25){\framebox(5,5)  {$\bullet$}  }

\put(45,0){\line(0,1){5}}
\put(45,5){\framebox(5,5)  {$\leftarrow$} }
\put(45,10){\framebox(5,5)  {$\uparrow$} }
 \put(50,10){\framebox(5,5)    }

 \put(45,15){\framebox(5,5)    }
 \put(50,15){\framebox(5,5)  {$\leftarrow$}  }

 \put(45,20){\framebox(5,5)    }

 \put(50,20){\framebox(5,5)    }
 \put(55,20){\framebox(5,5)  {$\uparrow$}  }

 \put(45,25){\framebox(5,5)    }

 \put(50,25){\framebox(5,5)  {$\uparrow$}  }
 \put(55,25){\framebox(5,5)    }
  \put(60,25){\framebox(5,5)  {$\uparrow$}  }\put(65,25){\framebox(5,5)  {$\uparrow$}  }
  \end{picture}
\end{center}
\caption{  A tree-like tableau $T$ (left) and  its corresponding alternative tableau  $\beta(T)$ (right).}\label{alternative1}
\end{figure}
 The inverse map $\beta^{-1}$ is defined as follows.  Let $T'$   be an alternative tableau of size  $n$ with underlying diagram $F'$.  Suppose that $F'$ has $k$ rows and $n-k$ columns.  We first construct the underlying diagram $F$ of  $\beta^{-1}(T')$ from $F'$ by adding a column of $k$ cells to the left of the leftmost column. It is easily seen that $F$ is of size $n+1$. Next,    we fill the leftmost cell of row $i$ of $F$ with a point if and only if   row $i$ of $F'$ contains no $\leftarrow's$. Finally, fill the cell $(i,j)$ of $F$ with a point if and only if   the cell $(i,j)$ of $F'$ has an arrow.
  Let $\beta^{-1}(T')$ be the resulting tableau.

 By similar arguments as for the bijection $\alpha$, we can get the following properties of  $\beta$ analogous to $\alpha$.

\begin{theorem}\label{thebeta}
Let $n\geq 1$. For any $T\in \mathcal{T}_n$ and $T'\in \mathcal{AT}^{*}_{n}$ with $\beta(T)=T'$, we have the following relations.
\begin{itemize}
\item[(1)] $left(T)=urr(T')-1$;
    \item[(2)] $noc(T)=noc(T')$;
        \item[(3)] $top(T)=top(T')$, where $top(T')$ is the number of arrows in the topmost row of $T'$.
\end{itemize}

\end{theorem}

From Theorem \ref{thebeta}, we  deduce that, for $n\geq 1$,
\begin{equation}\label{eq2.1}
\sum_{T\in \mathcal{T}_n}a^{top(T)}b^{left(T)}=\sum_{T\in \mathcal{AT}^{*}_n}a^{top(T)}b^{urr(T)-1}
 \end{equation}
 and
 \begin{equation}\label{eq2.2}
\sum_{T\in \mathcal{T}_n}noc(T)a^{top(T)}b^{left(T)}=\sum_{T\in \mathcal{AT}^{*}_n}noc(T)a^{top(T)}b^{urr(T)-1}.
 \end{equation}

Let $n$ be a positive integer.
Combining   (\ref{eqT}) and (\ref{eq2.1}), we have
\begin{equation}\label{eq2.3}
 \sum_{T\in \mathcal{AT}^{*}_n}a^{top(T)}b^{urr(T)-1}=(a+b)_{n-1}.
 \end{equation}
Notice that  (\ref{eq2.3}) was first proved by   \cite{Cor1} using recurrence relations. Later,  \cite{Cor2} provided two bijective proofs of (\ref{eq2.3}).

 In order to prove Conjecture \ref{con2}, we need to introduce the notion of {\em   type $B$ alternative tableaux}, which are defined based on   shifted Ferrers  diagrams.  For a Ferrers diagram $F$ with $k$ columns, the {\em shifted} Ferrers diagram of $F$, denoted by $\bar{F}$, is the diagram obtained from $F$ by adding $k$ rows of size $1,2,\ldots, k$ above it in increasing order. The size of $\bar{F}$ is defined to be the size of $F$. In this context, the Ferrers diagram $F$ is called the subdiagram of $\bar{F}$. The {\em diagonal } of $\bar{F}$ is the set of rightmost cells in the added rows. A {\em diagonal } cell is a cell in the diagonal. We label the added rows as follows. If the diagonal cell of an added row is in column $i$, then the row is labeled with $-i$. The labels of the other rows and columns remain the same with $F$.  Figure \ref{shifted} illustrates a Ferrers diagram $F$ and its corresponding shifted Ferrers diagram $\bar{F}$, where the diagonal cells are marked with stars.

\begin{figure}[h!]
\begin{center}
\begin{picture}(150,230)
\setlength{\unitlength}{1.5mm}

\put(0,0){\line(0,1){5}   }
\put(0,5){\framebox(5,5)   }

\put(0,10){\framebox(5,5)   } \put(5,10){\framebox(5,5)   } \put(10,10){\framebox(5,5)   }

\put(0,15){\framebox(5,5)   } \put(5,15){\framebox(5,5)   } \put(10,15){\framebox(5,5)   }\put(15,15){\framebox(5,5)   }

\put(20,20){\line(1,0){10}   }\put(25,19.8){\circle*{0.5}  }

\put(-3, 1.5){$10$}\put(-2, 6.5){$8$}\put(-2, 11.5){$5$}\put(-2, 16.5){$3$}

\put(2,  21){$9$}\put(7,  21){$7$}\put(12,  21){$6$}\put(17,  21){$4$}
\put(22,  21){$2$}\put(27,  21){$1$}

\put(40,0){\line(0,1){5}   }
\put(40,5){\framebox(5,5)    }

\put(40,10){\framebox(5,5)   } \put(45,10){\framebox(5,5)    } \put(50,10){\framebox(5,5)   }

\put(40,15){\framebox(5,5)   } \put(45,15){\framebox(5,5)   } \put(50,15){\framebox(5,5)   }\put(55,15){\framebox(5,5)   }

 \put(40,20){\framebox(5,5)   } \put(45,20){\framebox(5,5)   } \put(50,20){\framebox(5,5)   }\put(55,20){\framebox(5,5)    }
 \put(60,20){\framebox(5,5)   }\put(65,20){\framebox(5,5) {*}  }

\put(40,25){\framebox(5,5)    } \put(45,25){\framebox(5,5)   } \put(50,25){\framebox(5,5)    }\put(55,25){\framebox(5,5)   } \put(60,25){\framebox(5,5)   {*}}

\put(40,30){\framebox(5,5)   } \put(45,30){\framebox(5,5)    }

\put(50,30){\framebox(5,5)   }
\put(55,30){\framebox(5,5)  {*} }

\put(40,35){\framebox(5,5)   } \put(45,35){\framebox(5,5)   }

\put(50,35){\framebox(5,5)  {*} }

\put(40,40){\framebox(5,5)    } \put(45,40){\framebox(5,5) {*}  }

\put(40,45){\framebox(5,5)  {*} }

\put(37, 1.5){$10$}\put(38, 6.5){$8$}\put(38, 11.5){$5$}\put(38, 16.5){$3$}
\put(36, 21.5){$-1$}\put(36, 26.5){$-2$}\put(36, 31.5){$-4$}\put(36, 36.5){$-6$}\put(36, 41.5){$-7$}\put(36, 46.5){$-9$}

\put(42,  51){$9$}\put(47,  51){$7$}\put(52,  51){$6$}\put(57,  51){$4$}
\put(62,  51){$2$}\put(67,  51){$1$}

  \end{picture}
\end{center}
\caption{  A Ferrers diagram  $F$ (left) and  its corresponding  shifted Ferrers diagram $\bar{F}$ (right).}\label{shifted}
\end{figure}
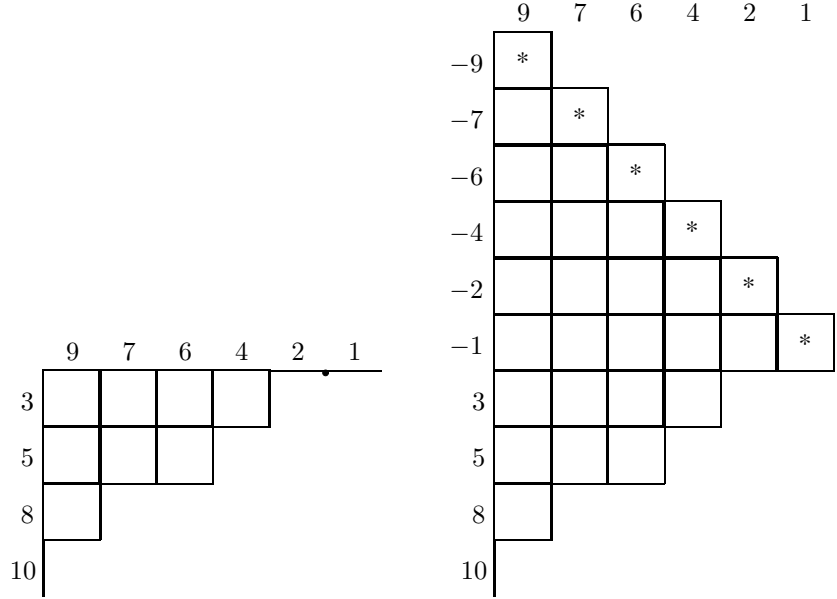

 A  {\em type $B$ alternative } tableau  is a  shifted Ferrers diagram with a partial filling of the cells with left arrows $\leftarrow$ and up arrows $\uparrow$ satisfying the following conditions:
\begin{itemize}
\item[(1)] all the cells to the left of  a  $\leftarrow$ in the same row, or above an  $\uparrow$ in the same column are empty;
    \item[(2)] if column $i$ contains an  $\uparrow$, then row $-i$ is empty;

        \item[(3)] the diagonal cells are empty.

\end{itemize}

The size of an alternative tableau of type $B$ is defined to be  the size of
the corresponding shifted Ferrers diagram.
Figure \ref{alternative} illustrates a type $B$  alternative tableau of
size $10$.
We denote by $\mathcal{AT}^{B}_n$ the set of alternative tableaux of type $B$ of size $n$.
In an alternative tableau of type $B$,   row $i$ is said to be {\em unrestricted } if and only if row $i$ has no $\leftarrow$ and column $|i|$, if it exists, has no $\uparrow$. Denote by $urr(T)$ the number of unrestricted rows of a type $B$ alternative tableau $T$. Analogous to the ordinary Ferrers diagram, a cell is called a {\em corner } of a shifted Ferrers diagram if both its bottom and right edges are border edges.  If a corner $c$ of a type $B$ alternative tableau is empty, then we say that $c$ is non-occupied. Let $noc(T)$ denote the number of non-occupied corners of a type $B$ alternative tableau $T$.

\begin{figure}[h!]
\begin{center}
\begin{picture}(300,240)
\setlength{\unitlength}{1.5mm}

\put(40,0){\line(0,1){5}   }
\put(40,5){\framebox(5,5){$\leftarrow$}   }

\put(40,10){\framebox(5,5)   } \put(45,10){\framebox(5,5) {$\leftarrow$}  } \put(50,10){\framebox(5,5)   }

\put(40,15){\framebox(5,5)   } \put(45,15){\framebox(5,5)   } \put(50,15){\framebox(5,5) {$\leftarrow$}  }\put(55,15){\framebox(5,5)   }

 \put(40,20){\framebox(5,5)   } \put(45,20){\framebox(5,5)   } \put(50,20){\framebox(5,5)   }\put(55,20){\framebox(5,5) {$\leftarrow$}  }
 \put(60,20){\framebox(5,5)   }\put(65,20){\framebox(5,5)   }

\put(40,25){\framebox(5,5) {$\leftarrow$}  } \put(45,25){\framebox(5,5)   } \put(50,25){\framebox(5,5)  {$\uparrow$} }\put(55,25){\framebox(5,5)   } \put(60,25){\framebox(5,5)    }

\put(40,30){\framebox(5,5)   } \put(45,30){\framebox(5,5) {$\leftarrow$}  }

\put(50,30){\framebox(5,5)   }
\put(55,30){\framebox(5,5)    }

\put(40,35){\framebox(5,5)   } \put(45,35){\framebox(5,5)   }

\put(50,35){\framebox(5,5)   }

\put(40,40){\framebox(5,5) {$\uparrow$}  } \put(45,40){\framebox(5,5)    }

\put(40,45){\framebox(5,5)   }

\put(37, 1.5){$10$}\put(38, 6.5){$8$}\put(38, 11.5){$5$}\put(38, 16.5){$3$}
\put(36, 21.5){$-1$}\put(36, 26.5){$-2$}\put(36, 31.5){$-4$}\put(36, 36.5){$-6$}\put(36, 41.5){$-7$}\put(36, 46.5){$-9$}

\put(42,  51){$9$}\put(47,  51){$7$}\put(52,  51){$6$}\put(57,  51){$4$}
\put(62,  51){$2$}\put(67,  51){$1$}

  \end{picture}
\end{center}
\caption{ A type $B$ alternative tableau.}\label{alternative}
\end{figure}
  In the following, we describe a bijection  $\gamma$ between the set $\mathcal{AT}^{sym}_{2n}$ and the set $\mathcal{AT}^B_n$.
 For any symmetric alternative tableau $T\in \mathcal{AT}^{sym}_{2n}$, we can obtain a type $B$ alternative tableaux $\gamma(T)\in \mathcal{AT}^B_n$ by removing  all the cells strictly to the right of the main diagonal  of $T$. Conversely, given a type $B$ alternative tableaux $T'\in \mathcal{AT}^B_n$, we can recover a symmetric alternative tableau $\gamma^{-1}(T')\in \mathcal{AT}^{sym}_{2n} $ by adding the cells obtained by reflecting   the non-diagonal cells  across  its main diagonal.
  By the definition of the bijection $\gamma$, one can easily verify that $\gamma$ has the following properties.
 \begin{theorem}\label{thegamma}
 Let $n\geq 1$.
For any  $T \in \mathcal{AT}^{sym}_{2n}$ and $T'\in \mathcal{AT}^{B}_n$ with $\gamma(T)=T'$, we have the following ralations.
\begin{itemize}
\item[(1)] $urr(T)=urr(T')$;
    \item[(2)] $noc(T)=noc'(T')+2noc''(T')$, where $noc''(T')=noc(T')-noc'(T')$  and   $noc'(T')$ is the number of non-occupied diagonal corners of $T'$.
 \end{itemize}
 \end{theorem}

  Combining  Theorems \ref{thesym} and  \ref{thegamma}, we have
  \begin{equation}\label{eq2.4}
  \sum_{T\in \mathcal{T}^{sym}_{2n+1}}x^{left(T)-1}=\sum_{T\in \mathcal{AT}^{sym}_{2n}}x^{urr(T)-1}=\sum_{T\in \mathcal{AT}^{B}_{n}}x^{urr(T)-1},
  \end{equation}
 and
 \begin{equation}\label{eq2.5}
 \begin{array}{lll}
  \sum_{T\in \mathcal{T}^{sym}_{2n+1}}noc(T)x^{left(T)-1}&=&\sum_{T\in \mathcal{AT}^{sym}_{2n}}noc(T)x^{urr(T)-1}\\
  &=&\sum_{T\in \mathcal{AT}^{B}_{n}}(noc'(T)+2noc''(T))x^{urr(T)-1}.
  \end{array}
  \end{equation}

  In view of   (\ref{eqsT}) and (\ref{eq2.4}), we have
  \begin{equation}\label{eq2.6}
  AT^{B}_{n}(x)=\sum_{T\in \mathcal{AT}^{B}_{n}}x^{urr(T)-1}=2^n(x+1)_{n-1}.
  \end{equation}

\section{A bijection between $\mathcal{AT}^*_n$ and $\mathcal{L}_n$}
\label{sec:li}
In this section, we will present a bijection between alternative tableaux in which each column contains an $\uparrow$ and linked partitions. We begin with some definitions and notations.

Let $[n]=\{1,2,\ldots, n\}$. A {\em linked } partition of $[n]$ is a collection of nonempty subsets
$B_1, B_2, \ldots, B_k$ of $[n]$, called blocks, such that the union of $B_1$, $B_2$,$\ldots, $ $B_k$ is $[n]$ and any two distinct blocks are nearly disjoint. Two blocks $B_i$ and $B_j$ are said to be {\em nearly} disjoint if for any $k\in B_i\cap B_j$ one of the following conditions holds:
\begin{itemize}
\item[(1)] $k=min(B_i)$,  $|B_i|>1$, and $k\neq min(B_j)$;
\item[(2)] $k=min(B_j) $, $|B_j|>1$, and $k\neq min(B_i)$.
\end{itemize}

We denote by $\mathcal{L}_n$ the set of linked partitions of $[n]$.
   \cite{Chen1} introduced the linear representation of  a linked partition.
  For a linked partition of $[n]$, we arrange $n$ vertices  in a horizontal line and label them  with   $1,2,\ldots, n$ from left to right. For a block $B=\{b_1, b_2, \ldots, b_k\}$ with $k\geq 2$ and $b_1<b_2<\ldots<b_k$, we draw
an arc from $b_1$ to $b_j$ for  $j=2,\ldots ,k$. As an example, the linear representation of the linked partition $\{1,2,3,6\}, \{4,5\}, \{6,8\}, \{7,9\}, \{9,10\}, \{11\}$ is illustrated in Figure \ref{bijection1}.
For $i < j$, we use a pair $(i,  j)$ to denote an arc from $i$ to $j$, and we call $i$ and $j$ the left-hand
endpoint and the right-hand endpoint of arc $(i,j)$, respectively. A vertex labeled with $i$ is called vertex $i$.   Notice that the linear representation of a linked partition can also be viewed as   a simple graph  on $[n]$ such that for each vertex $i$ there is at most
one vertex $j$ such that $1\leq j\leq i$ and $j$ is connected to $i$.

A {\em type $B$} linked partition of $[n]$ is a linked partition of $[n]$ in which each integer may be negated.
We denote by $\mathcal{L}^B_n$ the set of type $B$ linked partition of $[n]$.
The linear representation of a type $B$ linked partitions is the same as that of the  ordinary linked partition  except that the labels of the vertices  maybe negative and the labels are increasing from left to right.

From now on, we will only consider linear representations of the (type B) linked partitions unless otherwise stated.
Let $\tau$ be a linked partition of $[n]$. A vertex $t$ of $\tau$  is called an {\em origin} if it is only a left-hand endpoint, or a
{\em transient} if it is both a left-hand point and a right-hand endpoint, or a {\em singleton} if it is an
isolated vertex,  or a {\em destination} if it   is only a right-hand endpoint. For a destination $i_t$, a path
$(i_1, i_2, \ldots,i_t)$ with  $i_1<i_2<\ldots<i_t$ is said to be  a {\em maximal } path with destination $i_t$ if $i_1$ is an origin.
For example, in the linked partition given in Figure \ref{bijection1},   vertices $1,4,7$ are   origins,   vertices $6,9$ are   transients, vertices $2,3,5,8,10$ are destinations, vertex  $11$ is a singleton and (1,6,8) is a maximal  path with destination $8$.

Now   we define  a map $\Phi$: $\mathcal{AT}^{*}_n\rightarrow \mathcal{L}_n$.
Let $T$ be an    alternative tableau  of size $n$ without empty columns.    We construct a  linked partition $\tau=\Psi (T)$ as follows.  First, we arrange $n$ vertices in a line and label them by $1,2,\ldots, n$ from left to right.
 For column $j$,  suppose that the cell  $(i_1, j)$ is   filled with an $\uparrow$, and the  cells $(i_2,j)$, $(i_3, j)$, $\ldots, $ $(i_t, j)$ with  $i_2<i_3<\ldots<i_t$ are  cells filled with $\leftarrow's$. For $\ell=1,2,\ldots, t$, let $(i_{\ell}, i_{\ell+1})$ be an arc in $\tau$, where $i_{\ell+1}=j$. For example, Figure \ref{bijection1} demonstrates an alternative tableau $T\in \mathcal{AT}^{*}_{11}$ and its corresponding linked partition $\Phi(T)\in \mathcal{L}_{11}$.
Notice that there is at most one $\leftarrow$ in each row and exactly one $\uparrow$ in each column. This ensures that for any vertex $j$ in $\tau$, there is at most one arc whose right-hand endpoint is $j$. Hence, we have $\Phi(T)\in \mathcal{L}_n$.

\begin{figure}[h!]
\begin{center}
\begin{picture}(300,200)
\setlength{\unitlength}{1.5mm}

\put(0,0){\line(0,1){5}}
\put(0,5){\framebox(5,5)  {$\leftarrow$} }
\put(0,10){\framebox(5,5)  {$\uparrow$} }
 \put(5,10){\framebox(5,5)    }

 \put(0,15){\framebox(5,5)    }
 \put(5,15){\framebox(5,5)  {$\leftarrow$}  }

 \put(0,20){\framebox(5,5)    }

 \put(5,20){\framebox(5,5)    }
 \put(10,20){\framebox(5,5)  {$\uparrow$}  }

 \put(0,25){\framebox(5,5)    }

 \put(5 ,25){\framebox(5,5)  {$\uparrow$}  }
 \put(10,25){\framebox(5,5)    }
  \put(15,25){\framebox(5,5)  {$\uparrow$}  }\put(20,25){\framebox(5,5)  {$\uparrow$}  }

  \put(-3, 1.5){$11$} \put(-2, 6.5){$9$}\put(-2, 11.5){$7$}\put(-2, 16.5){$6$}

  \put(-2, 21.5){$4$}\put(-2, 26.5){$1$}

  \put(1.5, 31){$10$}\put(6.5, 31){$8$}\put(11.5, 31){$5$}\put(16.5, 31){$3$}
  \put(21.5, 31){$2$}

  \put(35,0){\circle*{0.5}}\put(40,0){\circle*{0.5}}\put(45,0){\circle*{0.5}}
   \put(50,0){\circle*{0.5}}\put(55,0){\circle*{0.5}}\put(60,0){\circle*{0.5}}
  \put(65,0){\circle*{0.5}}\put(70,0){\circle*{0.5}}\put(75,0){\circle*{0.5}}
  \put(80,0){\circle*{0.5}}\put(85,0){\circle*{0.5}}

  \qbezier(35,0)(37.5,2.5)(40,0)\qbezier(35,0)(40,5)(45,0)
  \qbezier(35,0)(47.5,12)(60,0)\qbezier(50,0)(52.5,5)(55,0)
  \qbezier(60,0)(65.5,5)(70,0)\qbezier(65,0)(70,5)(75,0)
  \qbezier(75,0)(77.5,4)(80,0)

  \put(34,-2.5){$1$}\put(39,-2.5){$2$}\put(44,-2.5){$3$}\put(49,-2.5){$4$}
  \put(54,-2.5){$5$}\put(59,-2.5){$6$}\put(64,-2.5){$7$}\put(69,-2.5){$8$}
  \put(74,-2.5){$9$}\put(79,-2.5){$10$}\put(84,-2.5){$11$}
  \end{picture}
\end{center}
\caption{ An alternative tableau $T$ (left) and its corresponding linked partition $\Phi(T)$ (right).}\label{bijection1}
\end{figure}
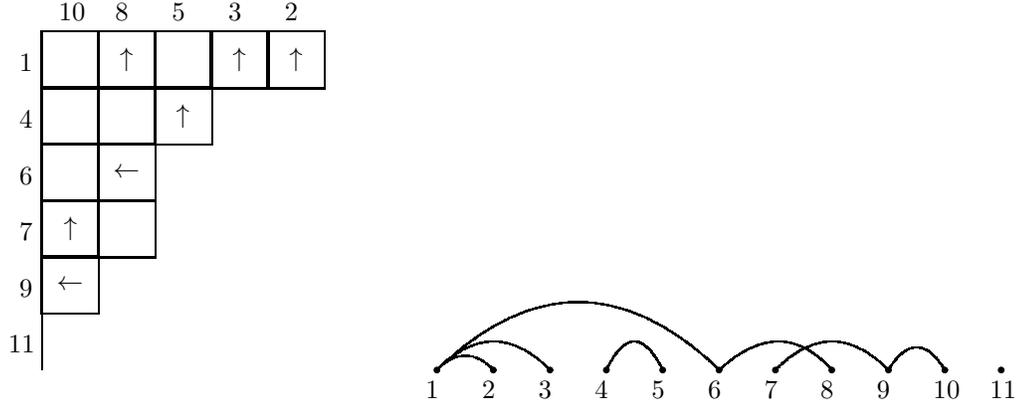

\begin{theorem}\label{thphi}
The map $\Phi :$  $\mathcal{AT}^{*}_n\rightarrow \mathcal{L}_n$ is a bijection.
\end{theorem}
\begin{proof} We will prove the assertion by constructing the inverse map of $\Phi$.
Now we define  a map $\Psi:$ $\mathcal{L}_n\rightarrow \mathcal{AT}^{*}_n$. Let $\tau$ be a linked partition of $[n]$. Let $F$ be a Ferrers diagram of size $n$ such that for all $j\in [n]$, $j$ is the label of a column if and only if vertex $j$ is a destination.

Suppose that $\tau$ has exactly $k$   destinations, which are  vertices $b_1, b_2, \ldots, b_k$ with $b_1>b_2>\ldots >b_k$. We set $T^{(0)}$ to be the empty filling of $F$ and $\tau^{(0)}=\tau$. For $m$ from $1$ to $k$, we define $T^{(m)}$ to be
the tableau obtained from  $T^{(m-1)}$ by filling the cells in column $b_m$ as follows.
       Suppose that
     the maximal   path in $\tau^{(m-1)}$ with   destination $b_m$ is $(a_1, a_2, \ldots, a_t,  b_m)$.
     Then we fill the cell  $(a_1,b_m)$ with an $\uparrow$. For all $2\leq j\leq t$,  fill the cell $(a_j, b_m)$ with a $\leftarrow$.  Let $\tau^{(m)}$ be  the linked partition obtained form $\tau^{(m-1)}$ by removing the arcs in the path $(a_1, a_2, \ldots, a_t, b_m)$.
Set $T=\Psi (\tau)=T^{(k)}$.

It is apparent that   each column  of $F$ is filled with an $\uparrow$ possibly along with some $\leftarrow's$. We claim that    the cells pointed by an arrow are empty in $T^{(m)}$ for all $0\leq m\leq k$.  If not, suppose that the cell $(a_s, b_p)$ is non-empty in $T^{(m)}$ for some $2\leq s\leq t$ and $p< m$. Then the arc $(a_{s-1}, a_s)$ would be in the maximal  path  with destination $b_p$.   According to the construction  of $\tau^{(p)}$, the arc  $(a_{s-1}, a_s)$ is not contained in $\tau^{(j)}$ for all $j\geq p$.  This contradicts  the fact that $\tau^{(m-1)}$ contains  the arc $(a_{s-1}, a_s)$.
Hence, the claim is proved. This implies that  $\Psi(\tau)\in \mathcal{AT}^*_n$. From the definition of the map $\Phi$, it is clear that $\Phi(T)=\tau$. This implies that the map $\Phi: \mathcal{AT}^{*}_n\rightarrow \mathcal{L}_n$ is a bijection with inverse $\Phi^{-1}=\Psi$. This completes the proof. \end{proof}

From the definition of the map $\Phi$, it is not difficult to see  that $\Phi$ has the following properties.

\begin{theorem}\label{phi}
For any $T\in \mathcal{AT}^{*}_n$ and $\tau\in \mathcal{L}_n$ with $\Phi(T)=\tau$, we have the following.
\begin{itemize}
\item[(1)] The labels of the columns of $T$ are exactly the labels of the destinations of $\tau$.
            \item[(2)] The labels of the rows of $T$ are exactly the labels of the origins, singletons and transients of $\tau$.
\item[(3)] The labels of the unrestricted rows of $T$ are exactly the labels of the origins and singletons of $\tau$.
    \item[(4)] The number of $\uparrow's$ in the topmost row of $T$ is equal to the number of arcs whose left-hand endpoints are  vertex $1$ in $\tau$.

\end{itemize}
\end{theorem}

 To conclude this section, we remark that
 \cite{Chen2} established a bijection between permutation tableaux of size $n$ and linked partitions of $[n]$.  The combination of their bijection  and  the bijection between alternative tableau of size $n$ in which each column contains an up arrow    and permutation tableaux of size $n$ established  by  \cite{Cor1} is in fact identical with our bijection $\Phi$.

\section{Proof of Conjecture \ref{con1}}
\label{sec:co1}
In this section, we aim to prove Conjecture \ref{con1}.
For  $2\leq i\leq n$, denote by $\mathcal{AT}^{*}_{n, i}$ the set of alternative tableaux  $T\in \mathcal{AT}^{*}_n$ such that the cell $(i-1,i)$ of $T$ is a non-occupied corner. Notice that for any $T\in \mathcal{AT}^{*}_n$, each column of $T$ contains an up  arrow. This implies that the cell $(1,2)$ cannot be a non-occupied corner of $T$. Hence we have $\mathcal{AT}^{*}_{n,2}=\emptyset$.

Recall that for an alternative tableau $T$, the weight of $T$ is defined to be $a^{top(T)}b^{urr(T)-1}$.
Analogously, for a linked partition $\tau$, we define  the {\em weight} of $\tau$    to be  $a^{one(\tau)}b^{os(\tau)-1}$, where $os(T)$ is the total  number of origins and singletons  of $\tau$ and $one(\tau)$ is the number of arcs whose left-hand endpoints are vertex $1$ in $\tau$. Let
$$
L_n(a,b)=\sum_{\tau\in \mathcal{L}_n}a^{one(\tau)}b^{os(\tau)-1}.
$$
By Theorem \ref{phi} and Formula (\ref{eq2.3}), we have
\begin{equation}\label{eq:Ln}
L_n(a,b)=\sum_{\tau\in \mathcal{L}_n}a^{one(\tau)}b^{os(\tau)-1}=\sum_{T\in \mathcal{AT}^{*}_n}a^{top(T)}b^{urr(T)-1}=(a+b)_{n-1}.
\end{equation}

By Theorems \ref{thphi} and \ref{phi}, the bijection $\Phi$ maps  an alternative tableau $T\in \mathcal{AT}^{*}_{n, i} $ to a linked partition $\tau=\Phi(T)$ satisfying that
\begin{itemize}
\item[(a1) ] $top(T)=one(\tau)$ and $urr(T)=os(\tau)$;
\item[(a2) ] vertex $i$ is a destination of $\tau$;
\item[(a3) ] vertex $i-1$ is not a destination of $\tau$;
 \item[(a4) ] $(i-1, i)$ is not an arc in $\tau$.
\end{itemize}
For all $2\leq i\leq n$, denote by $\mathcal{L}_{n,i}$ the set of   linked partitions of $n$   satisfying   conditions (a2), (a3) and  (a4).
 Thus, the map $\Phi$ is a weight preserving bijection between the set $\mathcal{AT}^{*}_{n, i}$ and the set $\mathcal{L}_{n,i}$. This yields that,
 for all $3\leq i\leq n$,
 \begin{equation}\label{tau8}
 \sum_{T\in \mathcal{AT}^{*}_{n, i} }a^{top(T)}b^{urr(T)-1}=\sum_{\tau\in \mathcal{L}_{n,i}}a^{one(\tau)}b^{os(\tau)-1}.
 \end{equation}

 \begin{lemma}\label{lemweightp}
 For $n\geq 3$, we have
 \begin{equation}\label{weightp}
   \sum_{i=3}^{n}\sum_{\tau\in \mathcal{L}_{n,i}}a^{one(\tau)}b^{os(\tau)-1}=\big[ (n-2)ab+{n-2\choose 2}(a+b)+{n-2\choose 3}\big]\cdot  L_{n-2}(a,b).
   \end{equation}
 \end{lemma}
 \begin{proof}
For $3\leq i\leq n$, let $\mathcal{M}_{n,i}$ denote  the set of linked partitions of $[n]$  satisfying conditions $(a2)$ and $(a4)$. Denote  by $N_{n,i}$   the set of linked partitions of $[n]$ in which  both vertex $i-1$ and vertex $i$ are destinations. It is easily seen that
\begin{equation}\label{tau1}
  \sum_{i=3}^n\sum_{\tau\in \mathcal{L}_{n,i}}a^{one(\tau)}b^{os(\tau)-1}=\sum_{i=3}^{n}\sum_{\tau\in \mathcal{M}_{n,i}}a^{one(\tau)}b^{os(\tau)-1}-\sum_{i=3}^n\sum_{\tau\in \mathcal{N}_{n,i}}a^{one(\tau)}b^{os(\tau)-1}.
 \end{equation}

 For any $\tau\in \mathcal{M}_{n,i} $, we can obtain a linked partition $\tau'\in \mathcal{L}_{n-1}$ by removing vertex $i$ and the arc whose right-hand endpoint is vertex $i$, and  relabeling the rest vertices by $1,2\ldots, n-1$. Conversely, given a linked partition $\tau'\in \mathcal{L}_{n-1}$, we can get a linked partition $\tau\in \mathcal{M}_{n,i} $  by inserting  a vertex immediately after vertex $i-1$, relabeling the vertices by $1,2\ldots, n$ from left to right,  and adjoining an arc from vertex $j$ to vertex $i$ for some $1\leq j\leq i-2$. It is apparent that $ one(\tau) = one(\tau')+1$ when $j=1$ and $one(\tau) = one(\tau')$ otherwise.  Moreover, we have $os(\tau)=os(\tau')$. Hence,  for    $3\leq i\leq n$, we have
  \begin{equation}\label{tau2}
   \sum_{\tau\in \mathcal{M}_{n,i}}a^{one(\tau)}b^{os(\tau)-1} =\sum_{\tau'\in \mathcal{L}_{n-1}}(i-3+a)\cdot a^{one(\tau')}b^{os(\tau')-1}=(i-3+a)\cdot L_{n-1}(a,b).
 \end{equation}
 Summing over $i$ from $3$ to $n$, we deduce that
 \begin{equation}\label{tau4}
 \begin{array}{lll}
 \sum_{i=3}^{n}\sum_{\tau\in \mathcal{M}_{n,i}}a^{one(\tau)}b^{os(\tau)-1}&= & \sum_{i=3}^{n}(i-3+a)\cdot L_{n-1}(a,b)\\
 &=& ({n-2\choose 2}+(n-2)a)\cdot L_{n-1}(a,b).
 \end{array}
 \end{equation}
 From (\ref{eq:Ln}) we have $L_{n-1}(a,b)=(a+b+n-3)L_{n-2}(a,b)$.
 Substituting it into the above formula, we get
 \begin{equation}\label{tau6}
 \sum_{i=3}^{n}\sum_{\tau\in \mathcal{M}_{n,i}}a^{one(\tau)}b^{os(\tau)-1}=\big({n-2\choose 2}+(n-2)a\big )\cdot (a+b+n-3)\cdot L_{n-2}(a,b).
 \end{equation}

 Let   $3\leq i\leq n$. For any $\tau\in \mathcal{N}_{n,i} $, we can obtain a linked partition $\tau'\in \mathcal{L}_{n-2}$ by removing vertices $i-1$ and $i$ and the arcs whose right-hand points are vertices $i-1$ and $i$, and  relabeling the vertices by $1,2\ldots, n-2$. Conversely, given a linked partition $\tau'\in \mathcal{L}_{n-2}$, we can get a linked partition $\tau\in \mathcal{N}_{n,i} $  by inserting  two vertices  immediately after vertex $i-2$, relabeling the vertices by $1,2\ldots, n$,  and adjoining an arc from vertex $j_1$ (resp. $j_2$) to vertex $i-1$   (resp. vertex  $i$) for some $1\leq j_1, j_2\leq i-2$. It is apparent that
 $$
 one(\tau)=\left\{
  \begin{array}{ll}
  one(\tau')+2 &  \,\,  \mbox{if}\,  j_1=j_2=1 ;\\
 one(\tau')+1 &\,\, \mbox{if either }\,  j_1=1\, \mbox{or } j_2=1, \mbox{but not   both};\\
    one(\tau')&\,\,  \mbox{if} \, j_1, j_2\neq 1 .
  \end{array}
  \right.
  $$
  Moreover, we have $os(\tau)=os(\tau')$.
    Hence,  for  $3\leq i\leq n$, we have
  \begin{equation}\label{tau3}
  \begin{array}{lll}
   \sum_{\tau\in \mathcal{N}_{n,i}}a^{one(\tau)}b^{os(\tau)-1}&=&\sum_{\tau'\in \mathcal{L}_{n-2}}\big((i-3)^2+2a(i-3)+a^2\big)\cdot a^{one(\tau')}b^{os(\tau')-1}\\
   &=&((i-3)^2+2a(i-3)+a^2) L_{n-2}(a,b).
   \end{array}
 \end{equation}

 Summing over $i$ from $3$ to $n$, we deduce that
 \begin{equation}\label{tau7}
 \begin{array}{lll}
 \sum_{i=3}^{n}\sum_{\tau\in \mathcal{N}_{n,i}}a^{one(\tau)}b^{os(\tau)-1}&= & \sum_{i=3}^{n}((i-3)^2+2a(i-3)+a^2)\cdot  L_{n-2}(a,b)\\
 &=& \big({(n-3)(n-2)(2n-5)\over 6}+(n-2)a^2+2{n-2\choose 2}a\big)\cdot L_{n-2}(a,b).
 \end{array}
 \end{equation}

 Combining  (\ref{tau1}),  (\ref{tau6}) and  (\ref{tau7}) and by simple computation,  we are led to  (\ref{weightp}). This completes the proof.  \end{proof}

Now we are in the position to complete the proof of Conjecture \ref{con1}.\\
 {\noindent\bf  Proof of Conjecture \ref{con1}.} For $n\geq 3$, we have
$$
\begin{array}{lll}
\sum_{T\in \mathcal{T}_n}noc(T)a^{top(T)}b^{left(T)}&=& \sum_{T\in \mathcal{AT}^{*}_n}noc(T)a^{top(T)}b^{urr(T)-1}\,\,\,   (\mbox{by   (\ref{eq2.2})})\\
&=& \sum_{i=2}^{n}\sum_{T\in \mathcal{AT}^{*}_{n,i}} a^{top(T)}b^{urr(T)-1} \end{array}.$$
Recall that $\mathcal{AT}^{*}_{n,2}=\emptyset$. This yields that
$$
\sum_{T\in \mathcal{T}_n}noc(T)a^{top(T)}b^{left(T)}= \sum_{i=3}^{n}\sum_{\tau \in \mathcal{L}_{n,i}}a^{one(\tau)}b^{os(\tau)-1}\,\,\, (\mbox{by   (\ref{tau8})}).
 $$
 By Lemma \ref{lemweightp}, we have
 $$
   \sum_{T\in \mathcal{T}_n}noc(T)a^{top(T)}b^{left(T)} =\big[ (n-2)ab+{n-2\choose 2}(a+b)+{n-2\choose 3}\big]\cdot  L_{n-2}(a,b).
   $$
   Recall that $L_n(a,b)=(a+b)_{n-1}=T_{n}(a,b)$. This  yields that$$
   \sum_{T\in \mathcal{T}_n}noc(T)a^{top(T)}b^{left(T)}=\big[ (n-2)ab+{n-2\choose 2}(a+b)+{n-2\choose 3}\big]\cdot  T_{n-2}(a,b)
   $$ as desired, which completes the proof.

\section{A bijection between $\mathcal{AT}^B_n$ and $\mathcal{L}^B_n$  }
\label{sec:lb}
 In this section, we  shall establish  a bijection between  type $B$ alternative tableaux of size $n$ and  type $B$ linked partitions of $[n]$.
  First we describe a map $\Phi_B : \mathcal{AT}^{B}_n \rightarrow\mathcal{L}^{B}_n$.
Let $T$ be a type $B$  alternative tableau of size $n$.   We construct a type $B$ linked partition $\tau=\Phi_B(T)$ as follows. For all $j\in [n]$, if  $j$ is the label of a column, then let   $j$  be the label of a vertex of $\tau$ if column $j$ contains an $\uparrow$, and let $-j$ be the label of a vertex of $\tau$, otherwise. If $j$ is the label of a row, then let $j$ be the label of  a vertex of $\tau$.
 For column $j$ with an $\uparrow$,
 suppose that the cell  $(i_1, j)$ is  filled with an $\uparrow$, and the  cells $(i_2,j)$, $(i_3, j)$, $\ldots, $ $(i_t, j)$
 with $i_2<i_3<\ldots<i_t$
 are  cells filled with $\leftarrow's$. For $\ell=1,2,\ldots, t$, let $(i_{\ell}, i_{\ell+1})$ be an arc in $\tau$, where $i_{t+1}=j$.
  For non-empty column $j$ without any $\uparrow's$, suppose that the cells  $(i_1,j)$, $(i_2, j)$, $\ldots, $ $(i_t, j)$
   with $i_1<i_2<\ldots<i_t$
   are  filled with $\leftarrow's$. For $\ell=1,2,\ldots, t$, let $(-j, i_{\ell})$ be an arc in $\tau$. As an example, for the type $B$ alternative tableau given in Figure \ref{alternative}, its corresponding type $B$ linked partition is shown in Figure \ref{bijection2}.

   It is not difficult to check that for all $i\in [n]$, either vertex   $i$ or vertex  $-i$ is  contained in $\tau$, but not both. Notice that there is at most one $\leftarrow$ in each row and at most one $\uparrow$ in each column. Thus, for any vertex $j$ in $\tau$, there is at most one arc whose right-hand endpoint is $j$. In order to show that the map $\Phi_B$  is well defined, that is, $\Phi_B(T)\in \mathcal{L}^{B}_n$, it remains to show that   for $\ell=1,2,\ldots, t$,   vertex $i_{\ell}$ is contained in $\tau$. If $i_{\ell}>0$, vertex $i_{\ell}$ is contained in $\tau$ since $i_{\ell}$ is the label of a row of $T$. If $i_{\ell}<0$, then it follows from the definition of type $B$ alternative tableaux that  column $|i_{\ell}|$ contains no $\uparrow$. By the definition of the map $\Phi_B$,   vertex $i_{\ell}$ is contained in $\tau$. Hence, the map $\Phi_{B}$ is well defined.

\begin{figure}[h!]
\begin{center}
\begin{picture}(100,100)
\setlength{\unitlength}{1.5mm}

  \put(0,0){\circle*{0.5}}\put(5,0){\circle*{0.5}}\put(10,0){\circle*{0.5}}
   \put(15,0){\circle*{0.5}}\put(20,0){\circle*{0.5}}\put(25,0){\circle*{0.5}}
  \put(30,0){\circle*{0.5}}\put(35,0){\circle*{0.5}}\put(40,0){\circle*{0.5}}
  \put(45,0){\circle*{0.5}}

\qbezier(0,0)(2.5,2.5)(5,0) \qbezier(0,0)(5,5)(10,0) \qbezier(0,0)(12.5,12)(25,0)\qbezier(5,0)(10,5)(15,0)
\qbezier(10,0)(15,5)(20,0)\qbezier(10,0)(22.5,15)(35,0)
\qbezier(20,0)(25,5)(30,0)
  \qbezier(35,0)(37.5,2.5)(40,0)

  \put(-2,-2.5){$-7$}\put(3,-2.5){$-4$}\put(8,-2.5){$-2$}\put(13,-2.5){$-1$}
  \put(19,-2.5){$3$}\put(24,-2.5){$5$}\put(29,-2.5){$6$}\put(34,-2.5){$8$}
  \put(39,-2.5){$9$}\put(44,-2.5){$10$}
  \end{picture}
\end{center}
\caption{The type $B$ linked partition corresponding to the type $B$ alternative tableau given in Figure \ref{alternative}.  }\label{bijection2}
\end{figure}
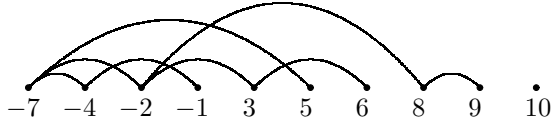

The map $\Phi_B: \mathcal{AT}^B_n\rightarrow \mathcal{L}^B_n$ is a generalization of the map $\Phi: \mathcal{AT}^*_n\rightarrow \mathcal{L}_n$ in the sense that for any  $T\in \mathcal{AT}^B_n $,  if all the cells in a row with a negative label are empty and each column is filled with an $\uparrow$, we can consider $T$ as an element in $\mathcal{AT}^*_n$ and the map $\Phi_B$ on such $T$ is identical with $\Phi$.

 In order to show that the map $\Phi_B$ is a bijection, we present a map $\Psi_B:\mathcal{L}^{B}_n\rightarrow \mathcal{AT}^{B}_n$.
  To this end, we need the following definitions.  Let   $\tau$ be a type $B$ linked partition of $[n]$.
   Suppose that $(i,j)$ is an arc of $\tau$ and $j$ is a destination of $\tau$.
    We say that $j$ is a {\em legal} destination if $|i|<|j|$. Otherwise, it is said to be {\em illegal}. As an example, in the linked partition given in  Figure \ref{bijection2},  vertex $6$ is a  legal destination, while vertex $5$ is an illegal destination.   Notice that if vertex $j$ is a legal destination, we must have $j>0$.  A
    path $(i_1, i_2, \ldots, i_t, i)$  with $i_1<i_2<\ldots<i_t<i$  are said to be a {\em good} path with   destination $i$ if we have $|i_j|<|i|$ for all $1\leq j\leq t$.  When we  choose $i_1$ to be the least such integer, the resulting path is said to be a {\em maximal} good path  with  destination $i$. For example, in the linked partition given in Figure \ref{bijection2},  the path $(-2,3, 6)$ is the maximal good path  with destination $6$.

   Now we  are ready  to  describe the map $\Psi_B$.  Given a type $B$ linked partition $\tau$ of $[n]$, let $F$ be a Ferrers diagram of size $n$ such that  $|i|$ is  the label of a column of $F$  if and only if  either $i<0$  or  $i$ is a  legal destination. Let $\bar{F}$ be the shifted Ferrers diagram of $F$.
    Suppose that $\tau$ has exactly $k$ legal destinations, which are  vertices $b_1, b_2, \ldots, b_k$ with $b_1>b_2>\ldots >b_k$. We set $T^{(0)}$ to be the empty filling of $\bar{F}$ and $\tau^{(0)}=\tau$. For $m$ from $1$ to $k$, we define $T^{(m)}$ to be a   tableau obtained  from $T^{(m-1)}$ by filling the cells in column $b_m$ as follows.
       Suppose that
     the maximal good path with   destination $b_m$  in $\tau^{(m-1)}$  is $(a_1, a_2, \ldots, a_t,  b_m)$. Then we fill the cell  $(a_1,b_m)$ with an $\uparrow$. For all $2\leq j\leq t$,  fill the cell $(a_j, b_m)$ with a $\leftarrow$.  Let $\tau^{(m)}$ be  the type $B$ linked partition obtained form $\tau^{(m-1)}$ by removing the arcs in the path $(a_1, a_2, \ldots, a_t, b_m)$.
Let $\Psi_B(\tau)$ be the  tableau obtained from $T^{(k)}$ by    filling  the  cell $(j,-i)$ of $F$ with  a $\leftarrow$ if and only if $(i,j)$ is  an arc of $\tau^{(k)}$.

\begin{lemma}
The map $\Psi_B$ is well defined, that is,
  we have $\Psi_B(\tau)\in \mathcal{AT}^{B}_n$ for any $\tau\in \mathcal{L}^{B}_n$.
\end{lemma}
\begin{proof} We proceed the proof by following the notations in the definition of $\Psi_B$.   First  we shall show that the map $\Psi_B$ is feasible. To this end,  it suffices to show that  (1) for all $1\leq i\leq t$, the cell $(a_i, b_m)$ is contained in $\bar{F}$  and (2) the cell $(j,-i)$ is contained in $\bar{F}$ whenever $\tau^{(k)}$ contains  the arc $(i,j)$. From the construction of $\bar{F}$, it is not difficult to check that $b_m$ is the label of a column and $a_i$ is the label of a row in $\bar{F}$.  By the definition of the maximal good path with destination $b_m$, we have $|a_i|<|b_m|=b_m$. Hence $\bar{F}$ contains the cell $(a_i, b_m)$.

To prove (2), we claim that if $(i,j)$ is an arc in $\tau^{(k)}$, then we have $i<0$ and $|j|>|i|$. If not, the arc $(i,j)$ will be contained in  the maximal good path of a legal destination of $\tau$. This contradicts the definition  of $\Psi_B$. Hence, the claim is proved.  This implies that  $-i$ is the label of a column in $\bar{F}$.  By the construction of $\bar{F}$, we have that  $j$ is the label of a row. Hence, the cell $(j,-i)$ is contained in $F$.

 Next we aim to show that $\Psi_B(\tau)\in \mathcal{AT}^{B}_n$.  It is apparent that  the diagonal cells of $F$ are empty in $\Psi_B(\tau)$.
In order to show that $\Psi_B(\tau)\in \mathcal{AT}^{B}_n$, it remains  to show
that $\Psi_B(\tau)$ verifies the following properties:
\begin{itemize}
\item[{\upshape (\rmnum{1})}]  all the cells pointed by an arrow  are empty;
\item[{\upshape (\rmnum{2})}]  if column $i$ contains an up arrow  $\uparrow$, then row $-i$ is empty.
\end{itemize}
 We claim that    the cells pointed by an arrow are empty in $T^{(m)}$ for all $0\leq m\leq k$.  If not, suppose that the cell $(a_s, b_p)$ is non-empty in $T^{(m)}$ for some $2\leq s\leq t$ and $p< m$. Then the arc $(a_{s-1}, a_s)$ would be in the maximal good path  with destination $b_p$ since $|a_{s-1}|< b_m < b_p $.   According to the construction  of $\tau^{(p)}$, the arc  $(a_{s-1}, a_s)$ is not contained in $\tau^{(j)}$ for all $j\geq p$.  This contradicts  the fact that $\tau^{(m-1)}$ contains  the arc $(a_{s-1}, a_s)$.
Hence, the claim is proved.

In order to prove that  $\Psi_B(\tau)$ verifies property (\rmnum{1}), it remains  to show that if $(i,j)$ is  an arc in $\tau^{(k)}$, then there exists no arrow which  is pointed by the  $\leftarrow$  filled in the cell $(j,-i)$. If not,    it follows from the construction of $\Psi_B(\tau)$  that there exists a cell $(j,b_\ell)$ with $1\leq \ell\leq k$   which is  filled with an arrow. Since $b_\ell>|i|$, the arc $(i,j)$ would be contained in the maximal good path with destination $b_\ell$, which is impossible. Thus, we have concluded that $\Psi_B(\tau)$ verifies property (\rmnum{1}).

From the construction of $\Psi_B(\tau)$, it is easily seen that if column $i$ is filled with an $\uparrow$, then vertex $i$ is a legal destination in $\tau$. This implies that    $-i$ is not a vertex of $\tau$.   Hence, row $-i$ is empty, that is, $\Psi_B(\tau)$ verifies property  (\rmnum{2}). This completes the proof. \end{proof}

\begin{lemma}\label{inverse}
Following the notations in the definition of $\Psi_B$, we have $\Phi_B(T)=\tau$, if $\Psi_B(\tau)=T$.
\end{lemma}
\begin{proof} From the definition of the map $\Phi_B$, it is easily seen that $\Phi_B(T)$ has the same set of arcs as that of $\tau$. In order to show that $\Phi_B(T)=\tau$, it remains show that $\Phi_B(T)$ has the same set of vertices as that of $\tau$. By the definition of the map $\Psi_B$,  for all $j\in [n]$,
   vertex $j$ is contained in $\tau$ if and only if either $j$ is the label of a column   containing an $\uparrow$  or $j$ is the label of a row in $T$. Moreover, vertex $-j$ is contained in $\tau$  if and only if   $j$ is the label of  a column without any $\uparrow's$ in $T$. Thus, by the definition of the map $\Phi_B$, it follows that $\Phi_B(T)$ has the same set of vertices as that of $\tau$. This completes the proof. \end{proof}

By Lemma \ref{inverse}, we see that the map $\Psi_B$ is the inverse map of $\Phi_B$. Thus, the map $\Phi_B$ is indeed a bijection. From the definition of the map $\Phi_B$, one can easily verify that the bijection $\Phi_B$ satisfies  the following properties, which are analogous to $\Phi$.

\begin{theorem}\label{psi}
For any $T\in \mathcal{AT}^{B}_n$ and $\tau\in \mathcal{L}^{B}_n$ with $\Phi_B(T)=\tau$, we have the following properties.
\begin{itemize}
\item[1.] For all $i\in [n]$, $i$ is the label of a column of $T$ if and only if either vertex $i$ is a legal destination of $\tau$, or vertex $-i$ is contained in $\tau$.
            \item[2.] For all $i\in [n]$,  $i$ is the label of a row of $T$
            if and only if vertex $i$ is contained in $\tau$ and is not a legal destination of $\tau$.
\item[3.] The labels of the unrestricted rows of $T$ are exactly the labels of the origins and singletons of $\tau$.

\end{itemize}
\end{theorem}

For a type $B$ linked partition $\tau$, the {\em weight} of $\tau$ is defined to be  $x^{os(\tau)-1}$, where $os(T)$ is the total  number of origins and singletons  of $\tau$. Define
$$
L_n^B(x)=\sum_{\tau\in \mathcal{L}_n^B}x^{os(\tau)-1}.
$$
Combining  Theorem \ref{psi} and Formula (\ref{eq2.6}), we have
\begin{equation}\label{LB}
L_n^B(x)=\sum_{\tau\in \mathcal{L}_n^B}x^{os(\tau)-1}=\sum_{T\in \mathcal{AT}^{B}_n}x^{urr(T)-1}=2^n(x+1)_{n-1}.
\end{equation}

\section{Proof of Conjecture \ref{con2}}
\label{sec:co2}
This section is devoted to the proof of Conjecture \ref{con2}.
To this end, we need to consider a subset of
type $B$ alternative tableaux of size $n$.
For  $2\leq i\leq n$, denote by $\mathcal{AT}^{B}_{n, i}$ the set of type $B$ alternative tableaux of size $n$   in which the cell $(i-1,i)$ is a non-occupied corner.

From the construction of the map $\Phi_B$ and Lemma \ref{psi}, the map $\Phi_B$ maps  an type $B$ alternative tableau $T\in \mathcal{AT}^{B}_{n, i} $ to a type $B$ linked partition $\tau=\Phi_B(T)$ satisfying the following conditions:
\begin{itemize}
\item[(b1) ] $urr(T)=os(\tau)$;
\item[(b2) ]   either vertex $i$ is a legal destination of $\tau$, or vertex $-i$ is contained in $\tau$;
\item[(b3) ] vertex $i-1$ is contained in $\tau$  and is not a legal destination;
 \item[(b4) ] neither $(i-1, i)$ nor $(-i, i-1)$ is  an arc in $\tau$.
\end{itemize}
For all $2\leq i\leq n$, denote by $\mathcal{L}^{B}_{n,i}$ the set of all type $B$ linked partitions of $[n]$   satisfying   conditions (b2), (b3) and  (b4).
 Thus, the map $\Phi_B$ is a weight preserving bijection between the set $\mathcal{AT}^{B}_{n, i}$ and the set $\mathcal{L}^{B}_{n,i}$. This yields that
 for all $2\leq i\leq n$
 \begin{equation}\label{beq1}
 \sum_{T\in \mathcal{AT}^{B}_{n, i} }x^{urr(T)-1}=\sum_{\tau\in \mathcal{L}^{B}_{n,i}}x^{os(\tau)-1}.
 \end{equation}
In order to compute  the right-hand side of Formula $(\ref{beq1})$, we partition the set $\mathcal{L}^{B}_{n,i}$ into two subsets $\mathcal{X}_{n,i}$ and $\mathcal{Y}_{n,i}$, where $\mathcal{X}_{n,i}$ is the set of type $B$ linked partitions $\tau$ satisfying  conditions $(b3)$ and $(b4)$ in which vertex $-i$ is contained in $\tau$,  and $\mathcal{Y}_{n,i}$ is   the set of type $B$ linked partitions $\tau$ satisfying  conditions $(b3)$ and $(b4)$ in which  vertex $i$ is  a legal destination of  $\tau$.
It is apparent that  for $2\leq i\leq n$, we have
\begin{equation}\label{beq2}
 \sum_{\tau\in \mathcal{L}^{B}_{n,i}}x^{os(\tau)-1}=\sum_{\tau\in \mathcal{X}_{n, i} }x^{os(\tau)-1}+\sum_{\tau\in \mathcal{Y}_{n, i} }x^{os(\tau)-1}.
 \end{equation}

Denote by
$$L_n(x)=L_n(1,x)=\sum_{\tau\in \mathcal{L}_n}x^{os(\tau)-1}.$$
Recall that $L_n(a,b)=\sum_{\tau\in \mathcal{L}_n}a^{one(\tau)}b^{os(\tau)-1}=(a+b)_{n-1}$.
 Hence, we have
\begin{equation}\label{LN}
L_n(x)=(x+1)_{n-1}
\end{equation}

\begin{lemma}
For $n\geq 2$ and $2\leq i\leq n$, we have
\begin{equation}\label{beq3}
 \sum_{\tau\in \mathcal{X}_{n,i}}w(\tau)=2^{n-2}L_{n}(x)-(i-1)2^{n-2}L_{n-1}(x).
\end{equation}

\end{lemma}
 \begin{proof} For $2\leq i\leq n$,  denote by $\mathcal{X}^{(1)}_{n,i}$ the set of type $B$ linked partitions  $\tau$ of $[n]$ such that vertices $-i$ and $i-1$ are contained in $\tau$.  Let $\mathcal{X}^{(2)}_{n,i}$ denote the subset of type $B$ linked partitions $\tau \in  \mathcal{X}^{(1)}_{n,i}$ such that  vertex $i-1$ is a legal destination of $\tau$. Let $\mathcal{X}^{(3)}_{n,i}$ denote the subset of type $B$ linked partitions $\tau \in  \mathcal{X}^{(1)}_{n,i}$ such that    $\tau$ contains an arc $(-i, i-1)$.  Obviously, we have
\begin{equation}\label{beq4}
 \sum_{\tau\in \mathcal{X}_{n, i} }x^{os(\tau)-1} =\sum_{\tau\in \mathcal{X}^{(1)}_{n, i} }x^{os(\tau)-1}-\sum_{\tau\in \mathcal{X}^{(2)}_{n, i} }x^{os(\tau)-1}-\sum_{\tau\in \mathcal{X}^{(3)}_{n, i} }x^{os(\tau)-1}.
 \end{equation}
 For any type $B$ linked partition $\tau\in \mathcal{X}^{(1)}_{n, i} $, it can be uniquely  represented by an ordered pair $(\tau', I)$, where  $I$ is the vertex set of $\tau$ and $\tau'$ is an ordinary linked partition obtained from $\tau$ by relabeling the vertices of $\tau$ by $1,2,\ldots, n$ from left to right.  Obviously, we have $\{-i, i-1\}\subset I$.  Conversely, given such an ordered pair $(\tau', I)$, we can recover a type $B$ linked partition $\tau\in \mathcal{X}^{(1)}_{n, i} $  by reversing the above procedure.
Moreover, the total  number of origins and singletons of $\tau$ is equal to  that of $\tau'$.

Given a linked partition $\tau'\in \mathcal{L}_n$,  there are $2^{n-2}$ different ways to choose a set $I$ such that $\{-i, i-1\}\subset I$ to construct a
type $B$ linked partition $\tau\in \mathcal{X}^{(1)}_{n, i}$. Hence, we have
\begin{equation}\label{beq5}
 \sum_{\tau\in \mathcal{X}^{(1)}_{n, i} }x^{os(\tau)-1}=2^{n-2}\sum_{\tau'\in \mathcal{L}_n }x^{os(\tau')-1}=2^{n-2}{L_n(x)}.
 \end{equation}
 Let $\mathcal{Z}_{n,i}$ be the set of type $B$ linked partitions $\tau$ of $[n]$ such that vertex $-i$ is contained in $\tau$.
By  similar arguments as above, we have
\begin{equation}\label{beq6}
 \sum_{\tau\in \mathcal{Z}_{n,i} }x^{os(\tau)-1}=2^{n-1}\sum_{\tau'\in \mathcal{L}_n }x^{os(\tau')-1}=2^{n-1}L_n(x).
 \end{equation}

 For a type $B$ linked partition $\tau\in \mathcal{X}^{(2)}_{n, i} $, we can obtain a linked partition $\tau'\in \mathcal{Z}_{n-1, i-1}$ by removing vertex $i-1$ from $\tau$ and relabeling  vertex $j$  by $j-1$  and vertex $-j$ by $-j+1$ for all $j\geq i-1$. Conversely, given a  linked partition $\tau'\in \mathcal{Z}_{n-1, i-1}$, we can recover a linked partition $\tau\in \mathcal{X}^{(2)}_{n, i}$ from $\tau'$ by the following procedure.
  \begin{itemize}
  \item
    We first insert  a vertex labeled by $i-1$
      immediately after vertex $j$ where $j$ is the maximum  integer which is smaller than $i-1$;
  \item Then relabel vertex $\ell$ by $\ell+1$ and  vertex $-\ell$ by $-\ell-1$ for all $\ell>i-1$;
      \item Finally, adjoin an arc from vertex $s$ to  vertex $i-1$ for some $|s|<i-1$.
\end{itemize}
 It is not difficult to see that  the total  number of origins and singletons of $\tau$ remains  the same as that of  $\tau'$.  Hence, we have
 \begin{equation}\label{beq7}
 \sum_{\tau\in \mathcal{X}^{(2)}_{n,i} }x^{os(\tau)-1}=\sum_{\tau'\in \mathcal{Z}_{n-1, i-1}}(i-2)x^{os(\tau')-1}.
 \end{equation}
 From (\ref{beq6}), it follows that
 \begin{equation}\label{beq8}
 \sum_{\tau\in \mathcal{X}^{(2)}_{n,i} }x^{os(\tau)-1}=(i-2)2^{n-2}L_{n-1}(x).
 \end{equation}

  We claim that for any type $B$ linked partition $\tau\in \mathcal{X}^{(3)}_{n, i} $,  vertex $i-1$ is a destination of $\tau$. If not, then  according  to the  the definition of $\Phi_B$, the cell $(i-1, i)$ is filled with a $\leftarrow$ in $T$  with $\Phi_B(T)=\tau$ and  there exists    one additional non-empty cell, say $c$,   in row $i-1$. Recall that the cell $(i-1, i)$ is a corner of $T$. This implies that the cell $c$  is pointed by the $\leftarrow$ filled in the cell $(i-1,i)$, which contradicts the definition of alternative tableaux. Hence, the claim is proved.

  By   similar arguments in the proof of (\ref{beq7}), we can deduce that \begin{equation}\label{beq9}
 \sum_{\tau\in \mathcal{X}^{(3)}_{n,i} }x^{os(\tau)-1}=\sum_{\tau'\in \mathcal{Z}_{n-1, i-1}}x^{os(\tau')-1}.
 \end{equation}
 From (\ref{beq6}), we deduce that
 \begin{equation}\label{beq10}
 \sum_{\tau\in \mathcal{X}^{(3)}_{n,i} }x^{os(\tau)-1}=2^{n-2}L_{n-1}(x).
 \end{equation}
 Combining  (\ref{beq4}), (\ref{beq5}), (\ref{beq8})  and (\ref{beq10}),  we are led to the desired Formula (\ref{beq3}), completing the proof.\end{proof}

 \begin{lemma}
For $n\geq 2$ and $2\leq i\leq n$, we have
\begin{equation}\label{beqy3}
 \sum_{\tau\in \mathcal{Y}_{n,i}}x^{os(\tau)-1}=(i-2)2^{n-2}L_{n-1}(x)-(i-2)^2 2^{n-2}L_{n-2}(x)
\end{equation}
\end{lemma}
\begin{proof} For $2\leq i\leq n$,  denote by $\mathcal{Y}^{(1)}_{n,i}$ the set of type $B$ linked partitions  $\tau$ such that vertex
 $i-1$  is  contained in $\tau$, vertex $i$ is a legal destination of $\tau$, and $(i-1, i)$ is not an arc of $\tau$.  Let $\mathcal{Y}^{(2)}_{n,i}$ denote the subset of type $B$ linked partitions $\tau \in  \mathcal{Y}^{(1)}_{n,i}$ such that  vertex $i-1$ is a legal destination of $\tau$.   Obviously, we have
\begin{equation}\label{beqy4}
 \sum_{\tau\in \mathcal{Y}_{n, i} }x^{os(\tau)-1} =\sum_{\tau\in \mathcal{Y}^{(1)}_{n, i} }x^{os(\tau)-1}-\sum_{\tau\in \mathcal{Y}^{(2)}_{n, i} }x^{os(\tau)-1}.
 \end{equation}

 Analogous to  the proof of (\ref{beq8}), we can deduce that
 \begin{equation}\label{beqy5}
  \sum_{\tau\in \mathcal{Y}^{(1)}_{n, i} }x^{os(\tau)-1}= \sum_{\tau'\in \mathcal{Z}_{n-1, i-1}}(i-2)x^{os(\tau')-1}=(i-2)2^{n-2}L_{n-1}(x).
 \end{equation}

 For a type $B$ linked partition $\tau\in \mathcal{Y}^{(2)}_{n,i}$, we can obtain a type $B$ linked partition $\tau'\in \mathcal{L}^{B}_{n-2}$ by removing vertices $i-1$ and $i$, and relabeling   vertex $j$ by $j-2$ and vertex $-j$ by $-j+2$  for all $j>i$.
 Conversely, given a  linked partition $\tau'\in \mathcal{L}^{B}_{n-2}$, we can recover a type $B$ linked partition $\tau\in \mathcal{Y}^{(2)}_{n, i}$ from $\tau'$ by the following procedures.
 \begin{itemize}
  \item we first relabel vertex $j$ by $j+2$ and  vertex $-j$ by $-j-2$ for all $j\geq i-1$;
  \item Then we insert  a vertex labeled with $i-1$ and a vertex labeled with $i$  immediately after vertex $j$ where $j$ is the maximum  integer which is smaller than $i-1$;
      \item Finally, adjoin an arc from vertex $\ell$ ( resp. $s$) to  vertex $i-1$ (resp. $i$) for some $|\ell|<i-1$ (resp. $|s|<i-1$).
\end{itemize}

 It is not difficult to see that  the total  number of origins and singletons of $\tau$ is the same as that of  $\tau'$. Thus,
 we have
 \begin{equation}\label{beqy6}
 \begin{array}{lll}
 \sum_{\tau\in \mathcal{Y}^{(2)}_{n,i} }x^{os(\tau)-1}&=&\sum_{\tau'\in \mathcal{L}^{B}_{n-2}}(i-2)^{2}x^{os(\tau')-1}\\
 &=&(i-2)^2 L^{B}_{n-2}(x)\\
 &=&(i-2)^2 2^{n-2}L_{n-2}(x).
 \end{array}
 \end{equation}
 The last equality in the above formula   follows from   (\ref{LB}) and (\ref{LN}).
 Combining   (\ref{beqy4}), (\ref{beqy5}) and (\ref{beqy6}), we are led to (\ref{beqy3}). This completes the proof.\end{proof}

Denote by $\mathcal{AT}^{B}_{n, 1}$ the set of type $B$ alternative tableaux of size $n$   in which the cell $(-1,1)$ is a non-occupied corner. It is easy to check that the map $\Phi_B$ sends an alternative tableau $T\in \mathcal{AT}^{B}_{n, 1}$ to a type $B$ linked partition $\tau\in \mathcal{ L}^B_{n}$ in which    $-1$ is the label of a vertex.   Let $\mathcal{L}^{B}_{n,1}$ denote the set of such type $B$ linked partitions.
Hence, the map $\Phi_B$ is a weight preserving bijection between the set $\mathcal{AT}^{B}_{n, 1}$ and the set $\mathcal{L}^{B}_{n,1}$.
This yields that
   \begin{equation}\label{Bn1}
 \sum_{T\in \mathcal{AT}^{B}_{n, 1} }x^{urr(T)-1}=\sum_{\tau\in \mathcal{L}^{B}_{n,1}}x^{os(\tau)-1}.
 \end{equation}

By similar arguments as in the proof of Formula (\ref{beq5}), we can deduce the following result, and the proof is omitted.
\begin{lemma}
For $n\geq 1$, we have
\begin{equation}\label{beqL}
 \sum_{\tau\in \mathcal{L}^B_{n,1}}w(\tau)= 2^{n-1}L_{n}(x).
\end{equation}
\end{lemma}

In the following, we proceed to complete the proof of Conjecture \ref{con2}.

{\noindent \bf Proof of Conjecture \ref{con2}.}
In view of  (\ref{eq2.5}), we have

$$
 \begin{array}{lll}
  \sum_{T\in \mathcal{T}^{sym}_{2n+1}}noc(T)x^{left(T)-1}&=&\sum_{T\in \mathcal{AT}^{sym}_{2n}}noc(T)x^{urr(T)-1}\\
  &=&\sum_{T\in \mathcal{AT}^{B}_{n}}(noc'(T)+2noc''(T))x^{urr(T)-1}\\
  &=& \sum_{T\in \mathcal{AT}^{B}_{n,1}}x^{urr(T)-1} +2\sum_{i=2}^{n}\sum_{T\in \mathcal{AT}^{B}_{n,i}}x^{urr(T)-1}.
  \end{array}
  $$
From (\ref{Bn1}) and (\ref{beqL} ), we deduce that
  $$
  \sum_{T\in \mathcal{AT}^{B}_{n,1}}x^{urr(T)-1}=\sum_{\tau\in \mathcal{L}^{B}_{n,1}}x^{os(\tau)-1}=2^{n-1}L_{n}(x).
  $$

From (\ref{beq1}) and (\ref{beq2}), we deduce that
$$
\begin{array}{lll}
\sum_{i=2}^{n}\sum_{T\in \mathcal{AT}^{B}_{n,i}}x^{urr(T)-1}&=&\sum_{i=2}^{n}\sum_{\tau\in \mathcal{L}^{B}_{n,i}}x^{os(\tau)-1}\\
&=& \sum_{i=2}^{n}\sum_{\tau\in \mathcal{X}_{n, i} }x^{os(\tau)-1}+\sum_{i=2}^{n}\sum_{\tau\in \mathcal{Y}_{n, i} }x^{os(\tau)-1}.
\end{array}
$$

By (\ref{beq3}) and simple computation, we have
$$
\sum_{i=2}^{n}\sum_{\tau\in \mathcal{X}_{n, i} }w(\tau)=(n-1)2^{n-2}L_{n}(x)-{n\choose 2}2^{n-2}L_{n-1}(x).
$$
By (\ref{beqy3}) and simple computation, we have
$$
\sum_{i=2}^{n}\sum_{\tau\in \mathcal{Y}_{n, i} }w(\tau)={n-1\choose 2}2^{n-2}L_{n-1}(x)-{(n-2)(n-1)(2n-3)\over 6}2^{n-2}L_{n-2}(x)
$$
Substituting $L_n(x)=(x+1)_{n-1}$ and $T^{sym}_{2n+1}(x)=L^B_n(x)=2^nL_n(x)$ into the above formulae, we  conclude the proof of Conjecture \ref{con2}.

\section{Concluding notes}
In this paper, we obtain  the polynomial analogues of the number of non-occupied corners in tree-like tableaux and symmetric tree-like tableaux, confirming two conjectures posed by  \cite{Gao}.

In \cite{Lab}, he obtained the following results on the number of occupied corners in tree-like tableaux and symmetric tree-like tableaux.

\begin{theorem}{ \upshape   (See \cite{Lab}, Theorem 3.2)}\label{thoc1}
For $n\geq 1$, the number of occupied corners in tree-like tableaux of size $n$ is given by $n!$.
\end{theorem}

\begin{theorem}{ \upshape   (See \cite{Lab}, Theorem 3.7)}\label{thoc2}
For $n\geq 1$, the number of occupied corners in symmetric tree-like tableaux of size $2n+1$ is given by $2^n n!$.
\end{theorem}

Denote by  $oc(T)$  the number of occupied corners in (symmetric) tree-like tableau $T$.
As remarked in \cite{Gao},   we can adapt Laborde-Zubieta's easy proof of Theorems \ref{thoc1} and \ref{thoc2}  to deduce the following results by keeping track of left and top points.

\begin{theorem}
For $n\geq 1$, we have
$$
oc_n(a,b)=\sum_{T\in \mathcal{T}_n}oc(T)a^{top(T)} b^{left(T)}=T_n(a,b).
$$
\end{theorem}

\begin{theorem}
For $n\geq 1$, we have
$$
\sum_{T\in \mathcal{T}^{sym}_{2n+1}}oc(T)x^{left(T)-1}=T_{2n+1}^{sym}(x).
$$
\end{theorem}

\acknowledgements
\label{sec:ack}
The authors are very grateful to the referee
for  valuable  comments and suggestions.

\end{document}